\pdfoutput=1 
\documentclass[12pt, reqno]{amsart}

\usepackage[paper=letterpaper, inner=1in, outer=1in, top=1in, bottom=1in]{geometry}

\usepackage{booktabs, multirow} 
\usepackage{makecell} 
\usepackage{subcaption} 
\usepackage[flushmargin]{footmisc} 

\usepackage{mlmodern}
\usepackage[T1]{fontenc}
\usepackage{microtype} 

\usepackage[mathscr]{eucal}
\usepackage{amsfonts, amssymb, amsmath, mathtools}

\usepackage[backend=biber, style=alphabetic, maxbibnames=99, maxalphanames=5, url=false, isbn=false, doi=false]{biblatex}
\usepackage{xcolor}
\usepackage{hyperref}
\hypersetup{
    colorlinks=true,
    linkcolor=blue!50!black,
    urlcolor=teal!70!black,
    citecolor=green!50!black,
    breaklinks=true
}
\usepackage[capitalize,nameinlink,noabbrev,nosort]{cleveref}

\newbibmacro{string+doi}[1]{%
	\iffieldundef{doi}{\iffieldundef{url}{#1}{\href{\thefield{url}}{#1}}}{\href{http://dx.doi.org/\thefield{doi}}{#1}}}	
\DeclareFieldFormat[article]{title}{\usebibmacro{string+doi}{\mkbibquote{#1}}}
\DeclareFieldFormat[thesis]{title}{\usebibmacro{string+doi}{\mkbibemph{#1}}}

\usepackage{amsthm, thmtools}

\theoremstyle{plain}
\newtheorem{theorem}{Theorem}[section]
\newtheorem{lemma}[theorem]{Lemma}
\newtheorem{proposition}[theorem]{Proposition}
\newtheorem{corollary}[theorem]{Corollary}

\theoremstyle{definition}
\newtheorem{definition}[theorem]{Definition}
\newtheorem{example}[theorem]{Example}
\newtheorem{question}[theorem]{Question}

\theoremstyle{remark}
\newtheorem{remark}[theorem]{Remark}

\renewcommand*{\theHtheorem}{\theHsection.\the\value{theorem}}
\renewcommand*{\theHproposition}{\theHsection.\the\value{theorem}}
\renewcommand*{\theHlemma}{\theHsection.\the\value{theorem}}
\renewcommand*{\theHcorollary}{\theHsection.\the\value{theorem}}
\renewcommand*{\theHconjecture}{\theHsection.\the\value{theorem}}
\renewcommand*{\theHdefinition}{\theHsection.\the\value{theorem}}
\renewcommand*{\theHexample}{\theHsection.\the\value{theorem}}
\renewcommand*{\theHquestion}{\theHsection.\the\value{theorem}}
\renewcommand*{\theHremark}{\theHsection.\the\value{theorem}}

\usepackage{tikz}
\usetikzlibrary{calc, patterns, arrows, arrows.meta, decorations.pathreplacing}
\usepackage[centertableaux]{ytableau}
\ytableausetup{boxsize=1em}

\newcommand{\vocab}[1]{\textcolor{blue!50!black}{\emph{#1}}} 

\DeclareSymbolFont{sansops}{T1}{\sfdefault}{m}{n}
\SetSymbolFont{sansops}{bold}{T1}{\sfdefault}{b}{n}
\makeatletter
\renewcommand\operator@font{\mathgroup\symsansops}
\makeatother

\newcommand{\integers}{\mathbb{Z}}
\newcommand{\positives}{\integers_{>0}}
\newcommand{\nonnegatives}{\integers_{\geq 0}}
\newcommand{\rationals}{\mathbb{Q}}

\newcommand{\sym}[1]{\mathscr{S}_{#1}} 
\DeclareMathOperator{\adg}{\widehat{dg}}

\DeclareMathOperator{\Arm}{Arm}

\DeclareMathOperator{\leg}{leg}
\DeclareMathOperator{\arm}{arm}
\DeclareMathOperator{\Fill}{F}
\DeclareMathOperator{\NAF}{NAF}
\DeclareMathOperator{\rev}{rev}
\DeclareMathOperator{\inc}{inc}

\DeclareMathOperator{\content}{content}
\DeclareMathOperator{\twinv}{twinv}

\newcommand{\swap}[2][]{\mathfrak{t}_{#2}\ifx&#1&\else^{(#1)}\fi}
\newcommand{\sswap}[2][]{\mathfrak{t}_{#2}\ifx&#1&\else^{[#1]}\fi}
\newcommand{\esop}[2][]{\tilde\tau_{#2}\ifx&#1&\else^{(#1)}\fi}
\newcommand{\prob}[1]{\operatorname{prob}_{#1}}
\newcommand{\pp}[2][]{\rho_{#2}^{(#1)}}

\newcommand{\wt}[1][]{{\operatorname{wt}\ifx&#1&\else^{(#1)}\fi}}
\newcommand{\wtqt}[1][]{{\operatorname{wt}_{q, t}\ifx&#1&\else^{(#1)}\fi}}
\newcommand{\maj}[1][]{{\operatorname{maj}\ifx&#1&\else^{(#1)}\fi}}
\newcommand{\coinv}[1][]{{\operatorname{coinv}\ifx&#1&\else^{(#1)}\fi}}
\newcommand{\inv}[1][]{{\operatorname{inv}\ifx&#1&\else^{(#1)}\fi}}
\newcommand{\Des}[1][]{{\operatorname{Des}\ifx&#1&\else^{(#1)}\fi}}
\newcommand{\des}[1][]{{\operatorname{des}\ifx&#1&\else^{(#1)}\fi}}
\newcommand{\dg}[1][]{{\operatorname{dg}\ifx&#1&\else^{(#1)}\fi}}

\newcommand{\south}[1]{d(#1)}

\newcommand{\algebra}{A}

\newcommand{\interval}[1]{[#1]}

\usepackage{listofitems}
\newcommand{\composition}[2][]{%
	\def\temp{#2}\ifx\temp\empty%
	\varnothing%
	\else%
	\readlist\thecycle{#2}%
	\foreachitem\i\in\thecycle{\ifnum\icnt=1\else#1\fi\i}%
	\fi%
}

\newcommand{\window}[2][,]{%
	\left[
	\readlist\thecycle{#2}%
	\foreachitem\i\in\thecycle{\ifnum\icnt=1\else#1\fi\i}%
	\right]
}

\newcommand{\ind}[2][]{%
	\readlist\thecycle{#2}%
	\chi_{%
	\foreachitem\i\in\thecycle{\ifnum\icnt=1\else#1\fi\i}%
	}
}

\title[Shape changing identities for permuted-basement Macdonald polynomials]{Shape changing identities for permuted-basement nonsymmetric Macdonald polynomials \\ (Extended Abstract)}

\author{Guilherme Zeus Dantas e Moura}
\address{
    Department of Combinatorics and Optimization\\ 
    University of Waterloo\\
	Canada
}
\email{zeus@guilhermezeus.com, zeus.dantasemoura@uwaterloo.ca}

\author{Olya Mandelshtam}
\address{
    Department of Combinatorics and Optimization\\
    University of Waterloo\\
    Canada
}
\email{omandels@uwaterloo.ca}

\thanks{Both authors were partially supported by NSERC grant RGPIN-2021-02568.}


\begin{document}

\begin{abstract}
Permuted-basement Macdonald polynomials $E_\alpha^\sigma(\mathbf{x};q,t)$ are nonsymmetric generalizations of symmetric Macdonald polynomials that form a basis for the polynomial ring $\mathbb{Q}(q,t)[\mathbf{x}]$ for each fixed $\sigma$. There are combinatorial formulas for them as generating functions over composition-shaped non-attacking fillings. In this extended abstract, we bijectively prove identities for the relationship between $E_\alpha^\sigma$, $E_\alpha^{\sigma s_i}$, $E_{s_i\alpha}^\sigma$, and $E_{s_i\alpha}^{\sigma s_i}$. These identities correspond to two combinatorial operations on non-attacking fillings: (1) swapping adjacent entries in the \emph{basement}, generalizing a result of Alexandersson (2019), and (2) swapping adjacent parts in the \emph{shape}, which yields a straightening rule for expanding $E_\alpha^\sigma$ in the polynomials $\{E_{s_i\alpha}^\tau\}_\tau$.
\end{abstract}

\maketitle

\section{Introduction}

The \emph{nonsymmetric Macdonald polynomials}
\(E_\alpha(x_1, \ldots, x_n; q, t)\),
indexed by weak compositions \(\alpha=\composition{\alpha_1,\ldots,\alpha_n}\),
were introduced in \cite{Che95,Mac96,Opd95}
as a generalization of the symmetric Macdonald polynomials $P_\lambda$.
They are the monic eigenfunctions of the Cherednik--Dunkl operators
and form a basis of the polynomial ring \(\rationals(q, t)[x_1, \ldots, x_n]\).

The \emph{permuted-basement Macdonald polynomials}
\(E_\alpha^\sigma(x_1, \ldots, x_n; q, t)\),
first defined by \cite{Fer11},
introduce an additional level of indexing,
pairing a weak composition \(\alpha\) with a permutation \(\sigma\in \sym{n}\).
They are obtained from the nonsymmetric Macdonald polynomials
by applying sequences of Demazure--Lusztig operators.
Throughout, we fix $n \in \nonnegatives$ and work in the ring $\rationals(q,t)[\mathbf{x}]$
where $\mathbf{x}=x_1,\ldots,x_n$.
All weak compositions $\alpha$ are assumed to have $n$ parts, and all permutations $\sigma$ lie in $\sym{n}$. For $k \in \nonnegatives$, write $[k]\coloneq\{1,2,\ldots,k\}$.

There is a combinatorial interpretation for \(E_\alpha^\sigma\) as the generating function of non-attacking fillings of shape \(\alpha\) and basement \(\sigma\), due to \cite{Fer11} and as a generalization of \cite{HHL08}.
Our main results are \cref{theorem:main,theorem:three-term},
which we prove using bijective arguments:
for the former, we construct a probabilistic bijection on non-attacking fillings, and for the latter, we construct a pair of bijections on signed fillings.

\begin{theorem} \label{theorem:main}
	Let \(\alpha\) be a composition
	with \(\alpha_i=\alpha_{i+1}\) for some \(i\in\interval{n-1}\).
	Then,
	\begin{equation*}
		E_{\alpha}^{\sigma}(\mathbf{x}; q, t)
		= E_{\alpha}^{\sigma s_i}(\mathbf{x}; q, t).
	\end{equation*}
\end{theorem}

\begin{remark}
    \textcite[Theorem~22]{Ale19} proves \cref{theorem:main} for the case where \(\sigma_i = \sigma_{i+1} + 1\).
	We have not found an explicit reference for the result in full generality.
\end{remark}

\pagebreak[2]

\begin{theorem} \label{theorem:three-term}
	Let \(\alpha\) be a composition with \(\alpha_i > \alpha_{i+1}\) for some \(i \in \interval{n-1}\).
	Then, we have
	\begin{equation} \label{eq:three term}
		E_{\alpha}^{\sigma}(\mathbf{x}; q, t)
		=
		E_{s_i \alpha}^{\sigma s_i}(\mathbf{x}; q, t)
		+
		c_{i, \alpha, \sigma}(q, t)
		E_{s_i \alpha}^{\sigma}(\mathbf{x}; q, t),
	\end{equation}
	where
	\begin{equation*}
		c_{i, \alpha, \sigma}(q, t)
		=t^{\twinv(s_i \alpha, \sigma) - \twinv(\alpha, \sigma)}
		\frac{1 - t}{1 - q^{\leg(u) + 1} t^{\arm(u)}} \times \begin{cases}
			1 & \sigma_i>\sigma_{i+1}      \\
			q^{\leg(u) + 1} t^{\arm(u)}
			  & \sigma_i < \sigma_{i+1}\,,\end{cases}
	\end{equation*}
	in which \(u = (i + 1, \alpha_{i+1} + 1) \in \dg(s_i \cdot \alpha)\) and
	\begin{equation*}
		\twinv(\alpha, \sigma) = |\{(k, l) : k < l \text{ and } \alpha_k \geq \alpha_l \text{ and } \sigma_k < \sigma_l\}|.
	\end{equation*}
\end{theorem}

\begin{remark}
	\textcite[Proposition~17]{Ale19} proves \cref{theorem:three-term} for the case where \(\sigma_i = \sigma_{i+1} + 1\).
	The specialization for \(q=0\) is proved in \cite[Proposition~5.5]{AS19}.
	To the best of our knowledge, the identity \eqref{eq:three term} in the above generality is new.
\end{remark}

\begin{remark}
	In both identities with $\sigma s_i$ in the basement, $s_i$ acts on the \emph{right} of $\sigma$, swapping adjacent entries. This contrasts with the usual Demazure--Lusztig action $T_i E_\alpha^\sigma=t^* E_\alpha^{s_i \sigma}$ (see~\cite[Proposition 4.4.4]{Fer11}) where $s_i$ acts on the \emph{left}, and for $\sigma=id$ yields the known identity \cite[Eq. (8)]{HHL08}. Hence, our identities describe a new property for
	$E_\alpha^\sigma$.
\end{remark}

In \cref{sec:definitions}, we introduce the polynomials $E_\alpha^\sigma(\mathbf{x};q,t)$ as generating functions over non-attacking and signed fillings of composition shapes. We prove \cref{theorem:main} in \cref{sec:probabilistic} with a probabilistic bijection, and \cref{theorem:three-term} in \cref{sec:maj inv} using a $\maj$/$\inv$-preserving column-swapping bijection on signed fillings. We conclude in \cref{sec:straightening} with a straightening rule giving the coefficients in the expansion of $E_\alpha^\sigma$ in the basis $\{E_\beta^\tau\}$.

\subsection*{Acknowledgements.}
We thank Per Alexandersson, Sarah Mason, and Arun Ram for helpful discussions, and thank Donghyun Kim for directing us to \cite{KLO22}.

\section{Preliminaries} \label{sec:definitions}

A \vocab{composition} is a sequence
\(\alpha = \composition{\alpha_1, \alpha_2, \ldots, \alpha_n}\)
of nonnegative integers.
Define \(\rev(\alpha)=\composition{\alpha_n, \alpha_{n-1}, \ldots, \alpha_1}\), $\inc(\alpha)$ to be the weakly increasing rearrangement of the parts of $\alpha$.

Given a permutation \(\pi \in \sym{n}\) and a composition \(\alpha\),
the \vocab{left action} of \(\pi\) on \(\alpha\) is the composition
\(\pi \cdot \alpha = \composition{\alpha_{\pi^{-1}(1)}, \alpha_{\pi^{-1}(2)}, \ldots, \alpha_{\pi^{-1}(n)}}\).
This choice of notation guarantees that
\(\sigma \cdot (\pi \cdot \alpha) = (\sigma \pi) \cdot \alpha\)
for all \(\sigma, \pi \in \sym{n}\).
Note that \(\rev(\alpha) = w_0 \cdot \alpha\),
where \(w_0 = \window{n, n-1, \ldots, 1}\).
For permutations themselves, we use the usual product. In particular, \(\sigma s_i=\window{\sigma_1,\ldots,\sigma_{i+1},\sigma_i,\ldots,\sigma_n}\), that is, right multiplication by \(s_i\) swaps \(\sigma_i\) and \(\sigma_{i+1}\).

\subsection{Combinatorial formulas for \texorpdfstring{$E_\alpha^\sigma(\mathbf{x};q,t)$}{permuted-basement Macdonald polynomials}}

The \vocab{skyline diagram} of a composition
\(\alpha\)
is the set
$\dg(\alpha) = \left\{ (i, r) \,:\, i \in \interval{n} \text{ and } r \in \interval{\alpha_i} \right\}$,
where $(i,r)$ is a box in column $i$ and row $r$. The \vocab{augmented skyline diagram}
is the set $\adg(\alpha) = \dg(\alpha) \cup \left\{ (i, 0) \,:\, i \in \interval{n} \right\},$
where an zeroth row called the \vocab{basement} is added to the bottom of \(\dg(\alpha)\).

Let \(u = (i, r) \in \dg(\alpha)\).
The \vocab{arm set} of \(u\) is defined as
\begin{equation*}
	\Arm(u)  = \big\{
	(j, r-1) \in \adg(\alpha)
	\,:\,
	j < i \text{ and } \alpha_{j} < \alpha_{i}
	\big\} \cup \big\{
	(j, r) \in \dg(\alpha)
	\,:\,
	j > i \text{ and } \alpha_{j} \leq \alpha_{i}
	\big\}.
\end{equation*}
The \vocab{leg statistic} and \vocab{arm statistic} of \(u\) are defined as \(\leg(u) = \alpha_i-r\)
and \(\arm(u) = \left|\Arm(u)\right|\).
The box \vocab{south} of \(u = (i, r)\) is the box \(\south{u} = (i, r-1) \in \adg(\alpha)\).

A box \(u = (i_u, r_u)\) \vocab{attacks} \(v = (i_v, r_v)\) in \(\dg(\alpha)\) if
\(r_u = r_v\) and \(i_u < i_v\), that is, \(u\) and \(v\) are in the same row and \(u\) is to the left of \(v\); or
\(r_v = r_u - 1\) and \(i_v < i_u\), that is, \(u\) is in the row directly above \(v\) and \(u\) is to the right of \(v\), as in one of the configurations below:
\begin{center}
	\begin{tikzpicture}[x=.5cm, y=.5cm, baseline={(2,.5)}]
		\draw[black] (0, 0) rectangle (1, 1) node[midway] {\(u\)};
		\node at (2, .5) {\(\cdots\)};
		\draw[black] (3, 0) rectangle (4, 1) node[midway] {\(v\)};
	\end{tikzpicture}
	\hspace{3cm}
	\begin{tikzpicture}[x=.5cm, y=.5cm, baseline={(2,.5)}]
		\draw[black, shift={(0,-.5)}] (0, 0) rectangle (1, 1) node[midway] {\(v\)};
		\node at (2, .5) {\(\cdots\)};
		\draw[black, shift={(0,.5)}] (3, 0) rectangle (4, 1) node[midway] {\(u\)};
	\end{tikzpicture}
\end{center}

A \vocab{non-augmented filling} of shape \(\alpha\) is a map \(T \colon \dg(\alpha) \to \interval{n}\).
An \vocab{augmented filling} of shape \(\alpha\) and basement \(\sigma \in \sym{n}\) is
a map \(T \colon \adg(\alpha) \to \interval{n}\) such that \(T(i, 0) = \sigma_i\) for all \(i \in \interval{n}\).
A filling \(T\), augmented or not, is \vocab{non-attacking}
if \(T(u) \neq T(v)\) whenever box \(u\) attacks box \(v\).
Let $\Fill(\alpha, \sigma)$ and $\NAF(\alpha, \sigma)$ denote the sets of fillings and non-attacking fillings, respectively, of shape $\alpha$ with basement $\sigma$.

The \vocab{content} of a filling \(T\) is the composition
$\beta=\composition{\beta_1, \beta_2, \ldots, \beta_n}$,
where \(\beta_i = |\{u\in\dg(\alpha) : T(u)=i\}|\).
We write \(\mathbf{x}^T = x_1^{\beta_1} x_2^{\beta_2} \cdots x_n^{\beta_n}\).
Let \(\NAF(\alpha, \sigma, \beta)\) denote
the set of non-attacking fillings
of shape \(\alpha\),
basement \(\sigma\),
and content \(\beta\).

Let \(\Des(T) = \{ u \in \dg(\alpha) : T(u) > T(\south{u}) \}\) be the set of descents in \(T\).
The \vocab{major index} of \(T\) is \( \maj(T) = \sum_{u \in \Des(T)} \left(\leg(u) + 1\right) \).
See \cref{fig:major-index} for an example.
\begin{figure}[htbp]
	\centering
	\begin{tikzpicture}[x=.4cm, y=.4cm]
		\def\shape{2, 2, 0, 1}
		\def\filling{
			3, 1, \textbf2,
			1, \textbf2, \textbf4,
			2,
			4, 4}
		\xdef\reading{0}
		\fill[black!20] (1,2) rectangle ++(1,1);
		\fill[black!20] (2,1) rectangle ++(1,1);
		\fill[black!20] (2,2) rectangle ++(1,1);
		\foreach \height [count = \column] in \shape{
			\foreach \row in {0,...,\height}{
					\pgfmathparse{int(\reading + 1)}
					\xdef\reading{\pgfmathresult}
					\draw[black] (\column, \row) rectangle (\column + 1, \row + 1);
					\coordinate (B\reading) at (\column + .5, \row + .5);
				}
		}
		\foreach \i [count = \j] in \filling{
			\node at (B\j) {\i};
		}
	\end{tikzpicture}
	\caption{A non-attacking augmented filling \(T\) of shape \(\alpha = \composition{2, 2, 0, 1}\), basement \(\sigma = \window{3, 1, 2, 4}\), and \(\content(T) = \composition{1, 2, 0, 2}\). The descents in \(T\) are shaded, with legs \(\leg(1,2)=\leg(2,2)=0\) and \(\leg(2,1)=1\). Thus, \(\maj(T) = 1 + 1 + 2 = 4\).}
	\label{fig:major-index}
\end{figure}

We define the \vocab{inversion} and \vocab{coinversion} statistics,
following \cite[Section~3]{HHL08}. Given \(a, b \in \interval{n}\),
let \(\ind{a, b} = 1\) if \(a > b\), and \(\ind{a, b} = 0\) otherwise.
Given \(a, b, c \in \interval{n}\), define \( \ind{a, b, c} = \ind{a, b} + \ind{b, c} - \ind{a, c} \in \{0, 1\} \).
A \vocab{triple} in \(\adg(\alpha)\) is a tuple \((u, v, w)\) of boxes such that \(v \in \Arm(u)\) and \(w = \south{u}\).
There are two possible configurations, shown in \Cref{fig:triples}.

\begin{figure}[htbp]
	\centering
	\begin{subfigure}[t]{0.46\textwidth}
		\centering
		\begin{ytableau}
			u & \none & \none[\cdots] & \none & v \\
			w
		\end{ytableau}
		\caption*{(\textsc{Type I}) The column containing \(u\) and \(w\) is at least as high as the column containing \(v\).}
	\end{subfigure}\hspace{0.03\textwidth}
	\begin{subfigure}[t]{0.46\textwidth}
		\centering
		\begin{ytableau}
			\none & \none & \none         & \none & u \\
			v     & \none & \none[\cdots] & \none & w
		\end{ytableau}
		\caption*{(\textsc{Type II}) The column containing \(u\) and \(w\) is strictly higher than the column containing \(v\).}
	\end{subfigure}
	\vspace{-.5em}
	\caption{Configurations in which $(u, v, w)$ forms a triple.}
	\label{fig:triples}
\end{figure}

A triple \((u, v, w)\) is an \vocab{inversion (resp.~coinversion) triple} if \(\ind{T(u), T(v), T(w)} = 1\) (resp.~\(0\)), and \(\inv(T)\) (resp.~\(\coinv(T)\)) denotes the number of such triples in \(T\).

\begin{theorem}[{\cite[Definition~4.4.2]{Fer11}}] \label{thm:tableau-formula}
	The permuted-basement Macdonald polynomial is given by
	\begin{equation*}
		E^{\sigma}_{\alpha}(\mathbf{x}; q, t)
		=
		\mkern-15mu
		\sum_{T \in \NAF(\alpha, \sigma)}
		\mkern-15mu
		\wt(T),
		\quad \text{where} \quad
		\wt(T) = \mathbf{x}^T
		q^{\maj(T)} t^{\coinv(T)}
		\mkern-20mu
		\prod_{\substack{u \in \dg(\alpha) \\ T(u) \neq T(\south{u})}}
		\mkern-15mu
		\tfrac{1 - t}{1 - q^{1 + \leg(u)} t^{1 + \arm(u)}}.
	\end{equation*}
\end{theorem}

Adapting the strategy in \cite[Section~8]{HHL05} for the symmetric Macdonald polynomials, we derive a formula for \(E_\alpha^\sigma\) in terms of \emph{signed fillings} (also called superfillings).

We replace the alphabet \(\interval{n}\) with the signed alphabet \(\interval{n}_{\pm} = \{1 < \bar{1} < 2 < \bar{2} < \cdots < n < \bar{n}\}\), where \(\bar{i}\) denotes the negative of \(i\).
For \(a, b \in \interval{n}_{\pm}\), define \(|a|\) to be the unsigned value of \(a\), and define \(\ind{a, b} = 1\) if \(a > b\) or \(a = b\) are negative, and \(\ind{a, b} = 0\) otherwise.
A \vocab{signed augmented filling} of shape \(\alpha\) and basement \(\sigma \in \sym{n}\) is a map \(T \colon \adg(\alpha) \to \interval{n}_{\pm}\) such that \(T(i, 0) = \sigma_i\) for all \(i \in \interval{n}\).
We extend the definitions of descents, major index, triples, and coinversions to signed fillings by using the updated \(\ind{\cdot, \cdot}\) function for comparisons between pairs of entries.
Let \(\Fill_\pm(\alpha, \sigma)\) denote the set of signed augmented fillings of shape \(\alpha\) and basement \(\sigma\).
For $T\in\Fill_\pm(\alpha, \sigma)$, let $|T|\in\Fill(\alpha, \sigma)$ be the filling obtained by taking the absolute value of each entry in $T$, and let $\operatorname{neg}(T)$ be the number of negative entries in $T$.
Let \(\NAF_\pm(\alpha, \sigma) = \{T \in \Fill_\pm(\alpha, \sigma) : |T| \in \NAF(\alpha, \sigma)\}\) denote the set of signed non-attacking fillings.

\begin{lemma}
	\label{lem:signed}
	The permuted-basement Macdonald polynomial is given by
	\begin{equation} \label{eq:signed}
		E^{\sigma}_{\alpha}(\mathbf{x}; q, t)
		= \mkern-10mu \sum_{T \in \Fill_\pm(\alpha, \sigma)} \mkern-10mu
		\wt_\pm(T),
		\quad \text{where} \quad
		\wt_\pm(T) =
		\frac{\mathbf{x}^{|T|} q^{\maj(T)} t^{\coinv(T)} (-t)^{\operatorname{neg}(T)}}
		{\prod_{u \in \dg(\alpha)} (1 - q^{1 + \leg(u)} t^{1 + \arm(u)})}.
	\end{equation}
\end{lemma}

\begin{proof}[Proof sketch]
	Adapting~\cite[Section~5.1]{HHL05}, there exists a \(\wt_\pm\)-sign-reversing involution on \(\Fill_\pm(\alpha, \sigma)\) with signed non-attacking fillings as fixed points.
	Thus, the right-hand side of \eqref{eq:signed} equals \(\sum_{T \in \NAF_\pm(\alpha, \sigma)} \wt_\pm(T)\).
	We then group the signed non-attacking fillings \(T\) with equal \(|T|\), yielding a compact formula over unsigned non-attacking fillings.
	By noting that \(\wt(S) = \sum_{|T| = S} \wt_\pm(T)\) for any \(S \in \NAF(\alpha, \sigma)\),
	we have that \(\sum_{T \in \NAF_\pm(\alpha, \sigma)} \wt_\pm(T)\)
	equals the formula in \cref{thm:tableau-formula}.
\end{proof}

\subsection{Probabilistic bijections}

Probabilistic bijections,
also referred to as \emph{bijectivization} or \emph{coupling} by \cite{BP19},
generalize weight-preserving bijections, and allow one to prove equalities of generating functions for combinatorial sets that either differ in cardinality, or do not admit a direct weight-preserving bijection.
Our exposition follows~\cite{AF22}.

\begin{definition}
	Let \(\algebra\) be an algebra.
	Let \(\mathbf{T}\) and \(\mathbf{U}\) be sets.
	A \vocab{probability map} from \(\mathbf{T}\) to \(\mathbf{U}\) is
	a function \(\prob{} \colon \mathbf{T} \times \mathbf{U} \to \algebra\) such that for every \(T \in \mathbf{T}\), we have
	$\sum_{U \in\mathbf{U}} \prob{}(T, U) = 1$.
\end{definition}

\begin{definition}[Probabilistic bijection {\cite{AF22}}]
	Let \(\algebra\) be an algebra.
	Let \(\mathbf{T}\) and \(\mathbf{U}\) be finite sets
	equipped with weight functions
	\(\wt_{\mathbf{T}} \colon \mathbf{T} \to \algebra\) and
	\(\wt_{\mathbf{U}} \colon \mathbf{U} \to \algebra\).
	A \vocab{probabilistic bijection} between
	\( (\mathbf{T}, \wt_{\mathbf{T}}) \) and
	\( (\mathbf{U}, \wt_{\mathbf{U}}) \) is
	a pair of probability maps
	\(\prob{\mathbf{T}} \colon \mathbf{T} \times \mathbf{U} \to \algebra\) and
	\(\prob{\mathbf{U}} \colon \mathbf{U} \times \mathbf{T} \to \algebra\)
	such that for every \(T \in \mathbf{T}\) and \(U \in \mathbf{U}\), we have
	\begin{equation*}
		\wt_{\mathbf{T}}(T) \prob{\mathbf{T}}(T, U)
		= \wt_{\mathbf{U}}(U) \prob{\mathbf{U}}(U, T).
	\end{equation*}
\end{definition}

\begin{proposition} \label{proposition:equal-weight-sum}
	Let \(\prob{\mathbf{T}} \colon \mathbf{T} \times \mathbf{U} \to \algebra\)
	and \(\prob{\mathbf{U}} \colon \mathbf{U} \times \mathbf{T} \to \algebra\)
	form a probabilistic bijection between
	\( (\mathbf{T}, \wt_{\mathbf{T}}) \) and
	\( (\mathbf{U}, \wt_{\mathbf{U}}) \).
	Then,
	\begin{equation*}
		\sum_{T \in \mathbf{T}} \wt(T) = \sum_{U \in\mathbf{U}} \wt(U).
	\end{equation*}
\end{proposition}

\section{A probabilistic bijection for non-attacking fillings} \label{sec:probabilistic}

Because of the complexity of the weight of a non-attacking filling, there is no ordinary weight-preserving bijection between
\(\NAF(\alpha,\sigma)\) and \(\NAF(\alpha,\sigma s_i)\)
to prove \cref{theorem:main}. However, a probabilistic bijection overcomes this obstacle. This section is based on \cite{DM25}.

We define a map on non-attacking fillings
from which we will construct a probabilistic bijection between
\(\NAF(\alpha,\sigma)\) and \(\NAF(\alpha,\sigma s_i)\).
Our technique is an adaptation to non-attacking fillings of a deterministic operator introduced in \cite{LN12} and used in \cite{CHMMW22} to bijectively prove certain column swapping identities on $\Fill(\lambda)$,
and a generalization of the probabilistic operator defined in \cite{Man24} for fillings of partition shapes for the symmetric Macdonald polynomials.

Fix $\sigma\in\sym{n}$ and let $\alpha$ be a composition with $i\in[n-1]$ such that $\alpha_i=\alpha_{i+1}$.

\begin{definition}
	Given a filling \(T\) of shape \(\alpha\) and \(r \in \interval{0, \alpha_i}\),
	let \(\swap[r]{i}(T)\) be the filling obtained by
	swapping the entries in boxes \((i, r)\) and \((i+1, r)\) in \(T\).
	Let \(\sswap[0, r]{i} = \swap[r]{i} \circ \cdots \circ \swap[1]{i} \circ \swap[0]{i}\),
	be a sequence of swaps
	in the first \(r+1\) rows in the columns \(i\) and \(i+1\).
\end{definition}

Note that for any \(r\geq 0\), $t_i^{(r)}$ preserves the content of a filling, and if \(T \in \Fill(\alpha, \sigma)\),
then \(\sswap[0, r]{i}(T) \in \Fill(\alpha, \sigma s_i)\).
However, \(\swap[r]{i}\) does not in general preserve the non-attacking property, so it does not give a valid map from $\NAF(\alpha,\sigma)$ to $\NAF(\alpha,\sigma s_i)$. Instead, we will use $\swap[r]{i}$ to define a probabilistic bijection between the two sets, as follows.

\begin{definition}[{\(\pp[r]{i}\)}]
	Let \(T \in \NAF(\alpha, \sigma)\) be a non-attacking filling.
	Let \(r \in \interval{0, \alpha_i-1}\).
	We define \(\pp[r]{i}(T) \in \rationals(q, t)\).
	Let \(a = T(i, r)\), \(b = T(i+1, r)\), \(c = T(i, r+1)\), and \(d = T(i+1, r+1)\).
	Let \(A = \arm(i+1, r+1)\) and \(\ell = \leg(i+1, r+1)\).
	Then, the value of \(\pp[r]{i}(T)\) is given in \cref{table:propagate}.
\begin{table}[htbp]
	\centering
	\begin{tabular}[c]{cccc}
		\toprule
		condition & sub-condition & \(\pp[r]{i}(T)\) & \(1 - \pp[r]{i}(T)\) \\
		\midrule
		\multirow{2}{*}[-.5ex]{\(|\{a, b, c, d\}| = 4\)} & \(\ind{c, d, a} = \ind{c, d, b}\)
		& \(0\) & \(1\) \\ \cmidrule{2-4}
		& \(\ind{c, d, a} = \ind{d, c, b}\)
		& \(1\) & \(0\) \\ \midrule
		& \(b = c\)
		& \(0\) & \(1\) \\ \cmidrule{2-4}
		\(|\{a, b, c, d\}| = 3\) & \(b = d\)
		& \(1\) & \(0\) \\ \cmidrule{2-4}
		& \(a = c\)
		& \(t^{1 - \ind{d, a, b}} \frac{1 - q^{\ell+1}t^{A+1}}{1-q^{\ell+1}t^{A+2}}\)
		& \((q^{\ell+1}t^{A+1})^{\ind{d, a, b}} \frac{1-t}{1 - q^{\ell+1}t^{A+2}}\) \\ \midrule
		\(|\{a, b, c, d\}| = 2\) & \(a = c\), \(b = d\)
		& \(1\) & \(0\) \\
		\bottomrule
	\end{tabular}
	\caption{The values of \(\pp[r]{i}(T)\) and \(1 - \pp[r]{i}(T)\) for \(T \in \NAF(\alpha, \sigma)\) and \(r \in \interval{0, \alpha_i-1}\).}
	\label{table:propagate}
\end{table}

\end{definition}

\begin{definition}[\(\prob{i}\)]
	The map \( \prob{i} \colon \NAF(\alpha, \sigma) \times \Fill(\alpha, \sigma s_i) \to \rationals(q, t) \)
	is defined for \(T \in \NAF(\alpha, \sigma)\) and
	\(U \in \Fill(\alpha, \sigma s_i)\) by letting
	\begin{equation*}
		\prob{i}(T, U) =
		\left( \prod_{r = 0}^{h-1} \pp[r]{i}(T) \right)
		\left( 1 - \pp[h]{i}(T) \right)
	\end{equation*}
	if there exists \(h \in \interval{0, \alpha_i}\)
	such that \(U = \sswap[0, h]{i}(T)\),
	and \(\prob{i}(T, U) = 0\) otherwise.
\end{definition}

\begin{proposition}
	The map
	\(\prob{i} \colon \NAF(\alpha, \sigma) \times \Fill(\alpha, \sigma s_i) \to \rationals(q, t)\)
	is a probability map, that is,
	for all \(T \in \NAF(\alpha, \sigma)\),
	\( \sum_{U \in \Fill(\alpha, \sigma s_i)} \prob{i}(T, U) = 1\).
\end{proposition}

The main result of this section is \cref{theorem:wtqt-prob}.
Define \(\wt(T) = \mathbf{x}^T \wtqt(T)\) for
\begin{equation*}
	\wtqt(T) = q^{\maj(T)} t^{\coinv(T)} \prod_{u \in \dg(\alpha),\ T(u) \neq T(\south{u})} \frac{1 - t}{1 - q^{1 + \leg(u)} t^{1 + \arm(u)}}.
\end{equation*}

\begin{theorem}[adapted from {\cite[Lemma 4.4]{Man24}}] \label{theorem:wtqt-prob}
	The pair of maps
	\begin{gather*}
		{\prob{i}} \colon \NAF(\alpha, \sigma, \beta) \times \NAF(\alpha, \sigma s_i, \beta) \to \rationals(q, t) \\
		{\prob{i}} \colon \NAF(\alpha, \sigma s_i, \beta) \times \NAF(\alpha, \sigma, \beta) \to \rationals(q, t)
	\end{gather*}
	defines a probabilistic bijection
	between \(\NAF(\alpha, \sigma, \beta)\) and \(\NAF(\alpha, \sigma s_i, \beta)\)
	with respect to the weight function \(\wtqt\).
	In other words,
	for all \(T \in \NAF(\alpha, \sigma, \beta)\)
	and \(U \in \NAF(\alpha, \sigma s_i, \beta)\),
	\begin{equation*}
		\wtqt(U)\prob{i}(U, T) = \wtqt(T)\prob{i}(T, U).
	\end{equation*}
\end{theorem}

\begin{example}
	Let \(n = 4\),
	\(\alpha = \composition{2, 2, 0, 1}\),
	\(\sigma = \window{3, 1, 2, 4}\),
	\(\beta = \composition{1, 2, 0, 2}\).
	Let \(i = 1\) and note that \(\alpha_1 = \alpha_2 = 2\).
	Moreover, note that \(\sigma s_i = \window{1, 3, 2, 4}\).
	\cref{theorem:wtqt-prob} implies that the maps
	\( \prob{1} \colon \NAF(\alpha, \sigma, \beta) \times \NAF(\alpha, \sigma s_1, \beta) \to \rationals(q, t)\)
	and
	\( \prob{1} \colon \NAF(\alpha, \sigma s_1, \beta) \times \NAF(\alpha, \sigma, \beta) \to \rationals(q, t)\)
	form a probabilistic bijection.

	We have \(
	\NAF(\alpha, \sigma, \beta) = \{
	\textcolor{red!50!black}{T_1},
	\textcolor{red!50!black}{T_2},
	\textcolor{red!50!black}{T_3}
	\}
	\), and \(
	\NAF(\alpha, \sigma s_i, \beta) = \{
	\textcolor{blue!50!black}{U_1},
	\textcolor{blue!50!black}{U_2}
	\}
	\),
	shown in the left and right sides of \cref{fig:prob-2201-3124-1022}, respectively.
	\cref{fig:prob-2201-3124-1022} shows the values of
	\(\prob{1}(T, U)\) and \(\prob{1}(U, T)\) for each pair of fillings.

	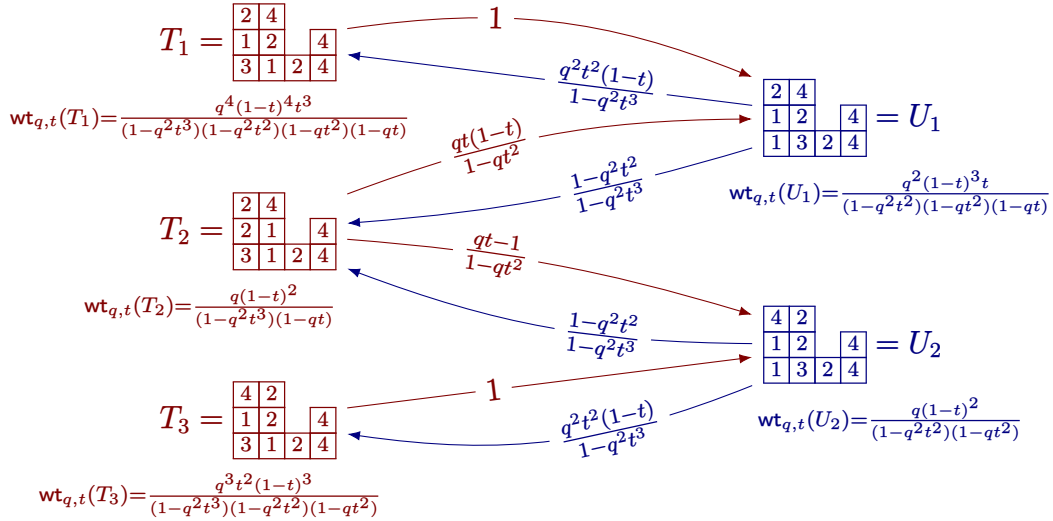
\begin{figure}[htbp]
		\centering
		\ytableausetup{smalltableaux}
		\begin{tikzpicture}
			\node[red!50!black]  (T1) at (-4,  2.5) {
				\(T_1 =
				\begin{ytableau}
					2 & 4 \\
					1 & 2 & \none & 4 \\
					3 & 1 & 2 & 4
				\end{ytableau}\)
			};
			\node[red!50!black]  (T2) at (-4,  0) {
				\(T_2 =
				\begin{ytableau}
					2 & 4 \\
					2 & 1 & \none & 4 \\
					3 & 1 & 2 & 4
				\end{ytableau}\)
			};
			\node[red!50!black]  (T3) at (-4, -2.5) {
				\(T_3 =
				\begin{ytableau}
					4 & 2 \\
					1 & 2 & \none & 4 \\
					3 & 1 & 2 & 4
				\end{ytableau}\)
			};
			\node[blue!50!black] (U1) at ( 4,  1.5) {
				\(\begin{ytableau}
					2 & 4 \\
					1 & 2 & \none & 4 \\
					1 & 3 & 2 & 4
				\end{ytableau} = U_1\)
			};
			\node[blue!50!black] (U2) at ( 4, -1.5) {
				\(\begin{ytableau}
					4 & 2 \\
					1 & 2 & \none & 4 \\
					1 & 3 & 2 & 4
				\end{ytableau} = U_2\)
			};
			\node[red!50!black, below, yshift=3pt, xshift=-.5cm] at (T1.south) {
				\(\scriptstyle \wtqt(T_1) =
				\frac{q^{4} {(1 - t)}^{4} t^{3}}{{(1 - q^{2} t^{3})} {(1 - q^{2} t^{2})} {(1 - q t^{2})} {(1 - q t)}}\)
			};
			\node[red!50!black, below, yshift=3pt, xshift=-.5cm] at (T2.south) {
				\(\scriptstyle \wtqt(T_2) =
				\frac{q {(1 - t)}^{2}}{{(1 - q^{2} t^{3})} {(1 - q t)}}\)
			};
			\node[red!50!black, below, yshift=3pt, xshift=-.5cm] at (T3.south) {
				\(\scriptstyle \wtqt(T_3) =
				\frac{q^{3} t^{2} {(1 - t)}^{3} }{{(1 - q^{2} t^{3})} {(1 - q^{2} t^{2})} {(1 - q t^{2})}}\)
			};
			\node[blue!50!black, below, yshift=3pt, xshift=.5cm] at (U1.south) {
				\(\scriptstyle \wtqt(U_1) =
				\frac{q^{2} {(1 - t)}^{3} t}{{(1 - q^{2} t^{2})} {(1 - q t^{2})} {(1 - q t)}}\)
			};
			\node[blue!50!black, below, yshift=3pt, xshift=.5cm] at (U2.south) {
				\(\scriptstyle \wtqt(U_2) =
				\frac{q {(1 - t)}^{2}}{{(1 - q^{2} t^{2})} {(1 - q t^{2})}}\)
			};
			\draw[-{Latex}, red!50!black] (T1)
			edge[bend left=15] node[sloped, fill=white, pos=.35]{\(1\)} (U1);
			\draw[-{Latex}, blue!50!black] (U1)
			edge[bend left=0] node[sloped, fill=white, pos=.35]{\(\tfrac{q^2t^2(1-t)}{1-q^2t^3}\)} (T1);
			\draw[-{Latex}, red!50!black] (T2)
			edge[bend left=10] node[sloped, fill=white, pos=.35]{\(\frac{qt(1-t)}{1-qt^2}\)} (U1);
			\draw[-{Latex}, blue!50!black] (U1)
			edge[bend left=6] node[sloped, fill=white, pos=.35]{\(\frac{1-q^2t^2}{1-q^2t^3}\)} (T2);
			\draw[-{Latex}, red!50!black] (T2)
			edge[bend left=6] node[sloped, fill=white, pos=.35]{\(\frac{qt-1}{1-qt^2}\)} (U2);
			\draw[-{Latex}, blue!50!black] (U2)
			edge[bend left=10] node[sloped, fill=white, pos=.35]{\(\frac{1-q^2t^2}{1-q^2t^3}\)} (T2);
			\draw[-{Latex}, red!50!black] (T3)
			edge[bend left=0] node[sloped, fill=white, pos=.35]{\(1\)} (U2);
			\draw[-{Latex}, blue!50!black] (U2)
			edge[bend left=15] node[sloped, fill=white, pos=.35]{\(\tfrac{q^2t^2(1-t)}{1-q^2t^3}\)} (T3);
		\end{tikzpicture}
		\caption{For $\alpha=2201$, $\sigma=[3,1,2,4]$, $\beta=1202$, we show \(T_i \in \NAF(\alpha,\sigma,\beta)\) and \(U_j \in \NAF(\alpha,\sigma s_1,\beta)\) and their $q,t$-weights. The label of an arrow from \(A\) to \(B\) is \(\prob{1}(A,B)\).}
		\label{fig:prob-2201-3124-1022}
	\end{figure}

	The fact that the maps form a probabilistic bijection can be checked
	by computing that, for each
	\(T \in \{\textcolor{red!50!black}{T_1}, \textcolor{red!50!black}{T_2}, \textcolor{red!50!black}{T_3}\}\)
	and each
	\(U \in \{\textcolor{blue!50!black}{U_1}, \textcolor{blue!50!black}{U_2}\}\),
	we have
	$\prob{1}(T, U) \wtqt(T) = \prob{1}(U, T) \wtqt(U)$.
	Moreover, using the fact that there exists a probabilistic bijection between
	\(\{\textcolor{red!50!black}{T_1}, \textcolor{red!50!black}{T_2}, \textcolor{red!50!black}{T_3}\}\)
	and
	\(\{\textcolor{blue!50!black}{U_1}, \textcolor{blue!50!black}{U_2}\}\),
	\cref{theorem:wtqt-prob} implies that their \(q, t\)-weight generating functions are equal:
	\begin{equation*}
		[x^\beta]E_\alpha^\sigma=\wtqt(\textcolor{red!50!black}{T_1})
		+ \wtqt(\textcolor{red!50!black}{T_2})
		+ \wtqt(\textcolor{red!50!black}{T_3})
		= \wtqt(\textcolor{blue!50!black}{U_1})
		+ \wtqt(\textcolor{blue!50!black}{U_2})=[x^\beta]E_\alpha^{\sigma s_i}.
	\end{equation*}
\end{example}

\cref{theorem:main} now follows from \cref{theorem:wtqt-prob} by applying \cref{proposition:equal-weight-sum} to obtain $[x^\beta]E_\alpha^\sigma=[x^\beta]E_\alpha^{\sigma s_i}$ for each content $\beta$.

\section{Bijection preserving major index and inversion number} \label{sec:maj inv}

To prove \cref{theorem:three-term}, we start by returning to the general set of fillings $\Fill(\alpha)$ and construct a weight-preserving bijection that exchanges two adjacent columns producing a filling in $\Fill(s_i\alpha)$. This bijection is equivalent to \cite[Proposition~4.3]{KLO22}, though we give a simplified presentation. In what follows, we work with fillings over an arbitrary totally ordered alphabet. Fix a composition $\alpha$ and an index $i\in[n-1]$ such that $\alpha_i>\alpha_{i+1}\coloneqq m$.

Our goal is to build a bijection that preserves row content and the statistics $\maj$ and $\inv$. Thus we must determine which rows swap entries when the two columns are exchanged, focusing on maintaining a ``local'' preservation of the $\inv$ statistic.  Since Type~I triples in columns $i,i+1$ of $T$ become Type~II triples in its image $\phi(T)$, we aim to match the inversion types of the corresponding triples between $T$ and $\phi(T)$. In many cases, for local preservation of $\inv$, the necessary swaps between consecutive rows are interdependent; we capture this dependence by grouping such rows into \vocab{blocks}, which will correspond to sets of adjacent rows that must be swapped simultaneously.

\begin{definition}
	For a filling \(T \in \Fill(\alpha)\) and $r\in\interval{m-1}$, we say that the rows \(r\) and \(r+1\) are \vocab{intertwined} if
	\( 	\ind{a, c, d} \neq \ind{b, c, d} \), where  \(a = T(i, r+1)\), \(b = T(i+1, r+1)\), \(c = T(i, r)\), and \(d = T(i+1, r)\). A \vocab{block} is a maximal consecutive set of intertwined rows. Equivalently, $\interval{m}$ decomposes uniquely into blocks such that consecutive rows are in different blocks if and only if they are not intertwined. See \Cref{fig:blocks}.
\end{definition}

\begin{figure}[htbp]
	\centering
	\begin{minipage}[b][][b]{.48\textwidth}
		\centering
		\begin{tikzpicture}[scale=.4]
			\node at (-0.5, 7.5) {\tiny \(7\)};
			\node at (-0.5, 6.5) {\tiny \(6\)};
			\node at (-0.5, 5.5) {\tiny \(5\)};
			\node at (-0.5, 4.5) {\tiny \(4\)};
			\node at (-0.5, 3.5) {\tiny \(3\)};
			\node at (-0.5, 2.5) {\tiny \(2\)};
			\node at (-0.5, 1.5) {\tiny \(1\)};
			\node[anchor=north] at (0.5, 1) {\(\scriptstyle i\)};
			\node[anchor=north] at (1.5, 1) {\(\scriptstyle i+1\)};
			\draw (0, 7) rectangle ++(1, 1);
			\draw (1, 6) rectangle ++(1, 1)[fill=gray!50];
			\draw (0, 6) rectangle ++(1, 1)[fill=gray!50];
			\draw (1, 5) rectangle ++(1, 1)[fill=gray!50];
			\draw (0, 5) rectangle ++(1, 1)[fill=gray!50];
			\draw (1, 4) rectangle ++(1, 1);
			\draw (0, 4) rectangle ++(1, 1);
			\draw (1, 3) rectangle ++(1, 1);
			\draw (0, 3) rectangle ++(1, 1);
			\draw (1, 2) rectangle ++(1, 1);
			\draw (0, 2) rectangle ++(1, 1);
			\draw (1, 1) rectangle ++(1, 1)[fill=gray!50];
			\draw (0, 1) rectangle ++(1, 1)[fill=gray!50];
			\node at (0.5, 7.5) {\(3\)};
			\node at (1.5, 6.5) {\(5\)};
			\node at (0.5, 6.5) {\(2\)};
			\node at (1.5, 5.5) {\(1\)};
			\node at (0.5, 5.5) {\(4\)};
			\node at (1.5, 4.5) {\(8\)};
			\node at (0.5, 4.5) {\(5\)};
			\node at (1.5, 3.5) {\(6\)};
			\node at (0.5, 3.5) {\(1\)};
			\node at (1.5, 2.5) {\(7\)};
			\node at (0.5, 2.5) {\(3\)};
			\node at (1.5, 1.5) {\(6\)};
			\node at (0.5, 1.5) {\(4\)};
			\draw[decorate, decoration={brace, amplitude=5pt, mirror, raise=5pt}] (2, 5.1) -- (2, 6.9) node[midway, anchor=west, xshift=10pt] {third block};
			\draw[decorate, decoration={brace, amplitude=5pt, mirror, raise=5pt}] (2, 2.1) -- (2, 4.9) node[midway, anchor=west, xshift=10pt] {second block};
			\draw[decorate, decoration={brace, amplitude=5pt, mirror, raise=5pt}] (2, 1.1) -- (2, 1.9) node[midway, anchor=west, xshift=10pt] {first block};
		\end{tikzpicture}
		\caption{Columns \(i\) and \(i+1\) of a filling \(T\) with three blocks.}
		\label{fig:blocks}
	\end{minipage}%
	\hspace{-.04\textwidth}%
	\begin{minipage}[b][][b]{.56\textwidth}
		\centering
		\begin{tikzpicture}[scale=.45]
			\draw (0, 0) rectangle ++(1, 1);
			\draw (1, 0) rectangle ++(1, 1);
			\draw[fill=gray!50] (0, 1) rectangle ++(1, 1);
			\draw[fill=gray!50] (1, 1) rectangle ++(1, 1);
			\fill[gray!50] (0, 2) rectangle ++(2, 2);
			\node at (1, 3.25) {\(\vdots\)};
			\draw[fill=gray!50] (0, 4) rectangle ++(1, 1);
			\draw[fill=gray!50] (1, 4) rectangle ++(1, 1);
			\draw (0, 5) rectangle ++(1, 1);
			\draw (1, 5) rectangle ++(1, 1);
			\node at (0.5, 5.5) {\(e\)};
			\node at (1.5, 5.5) {\(f\)};
			\node at (0.5, 4.5) {\(a\)};
			\node at (1.5, 4.5) {\(b\)};
			\node at (0.5, 1.5) {\(c\)};
			\node at (1.5, 1.5) {\(d\)};
			\node at (0.5, 0.5) {\(g\)};
			\node at (1.5, 0.5) {\(h\)};
		\end{tikzpicture}
		\caption{A block of a filling \(T\), with the entries \(a, b, c, d, e, f, g, h\) indicated.}
		\label{fig:general-block}
	\end{minipage}
\end{figure}

We construct a filling \(\phi(T) \in \Fill(s_i \cdot \alpha)\) derived from \(T\), starting by copying all entries in $T$ outside of columns $i$ and $i+1$. Then, for the upper rows $\interval{m+1,\alpha_i}$, we copy the entries from column $i$ of $T$ to column $i+1$ of $\phi(T)$. What remains is to determine the entries in columns $i,i+1$ in rows $[m]$ of $\phi(T)$; this reduces to choosing which blocks are swapped and which remain unchanged, where \emph{swapping a block} means swapping the entries between columns $i$ and $i+1$ within the block.

Focusing on a block of \(T\), let \((a,b)\) and $(c,d)$ be the entries  in columns \((i,i+1)\) in the top row and bottom row of the block, respectively.
Let \((e,f)\) and $(g,h)$ be the entries in columns \((i,i+1)\) in the row immediately above the block and in the row immediately below the block, respectively. $f$ may be undefined if it doesn't exist, and we set $g=h=\infty$ if the block is the bottommost block.
See \Cref{fig:general-block} for an illustration.

Then, we place the block in $\phi(T)$ by swapping the block if \(\ind{c, d, g} \neq \ind{e, a, b}\), and leaving it unchanged otherwise.
Non-intertwinedness ensures that the $\inv$ contributions of distinct blocks are independent; thus the total contribution to $\inv$ is preserved for each block. \Cref{fig:output-filling} shows the procedure with the filling from \Cref{fig:blocks} as the input.

The inversion number is preserved between $T$ and $\phi(T)$, by construction. Miraculously, the major index is preserved as well, which can be verified with a case analysis.

\begin{figure}[htbp]
	\centering
	\begin{tikzpicture}[scale=.47]
		\node at (-1.9,4.5) {$\phi(T)=$};
		\draw (1, 7) rectangle ++(1, 1);
		\draw (1, 6) rectangle ++(1, 1)[fill=gray!50];
		\draw (0, 6) rectangle ++(1, 1)[fill=gray!50];
		\draw (1, 5) rectangle ++(1, 1)[fill=gray!50];
		\draw (0, 5) rectangle ++(1, 1)[fill=gray!50];
		\draw (1, 4) rectangle ++(1, 1);
		\draw (0, 4) rectangle ++(1, 1);
		\draw (1, 3) rectangle ++(1, 1);
		\draw (0, 3) rectangle ++(1, 1);
		\draw (1, 2) rectangle ++(1, 1);
		\draw (0, 2) rectangle ++(1, 1);
		\draw (1, 1) rectangle ++(1, 1)[fill=gray!50];
		\draw (0, 1) rectangle ++(1, 1)[fill=gray!50];
		\node at (1.5, 7.5) {\(3\)};
		\node at (1.5, 6.5) {\(5\)};
		\node at (0.5, 6.5) {\(2\)};
		\node at (1.5, 5.5) {\(1\)};
		\node at (0.5, 5.5) {\(4\)};
		\node at (1.5, 4.5) {\(5\)};
		\node at (0.5, 4.5) {\(8\)};
		\node at (1.5, 3.5) {\(1\)};
		\node at (0.5, 3.5) {\(6\)};
		\node at (1.5, 2.5) {\(3\)};
		\node at (0.5, 2.5) {\(7\)};
		\node at (1.5, 1.5) {\(6\)};
		\node at (0.5, 1.5) {\(4\)};
		\draw[decorate, decoration={brace, amplitude=5pt, mirror, raise=5pt}] (2, 5.1) -- (2, 6.9) node[midway, anchor=west, xshift=10pt] {unchanged, since \(\ind{4, 1, 5} = 1 = \ind{3, 2, 5}\)};
		\draw[decorate, decoration={brace, amplitude=5pt, mirror, raise=5pt}] (2, 2.1) -- (2, 4.9) node[midway, anchor=west, xshift=10pt] {swapped, since \(\ind{2, 5, 8} = 0 \neq 1 = \ind{3, 7, 4}\)};
		\draw[decorate, decoration={brace, amplitude=5pt, mirror, raise=5pt}] (2, 1.1) -- (2, 1.9) node[midway, anchor=west, xshift=10pt] {unchanged, since \(\ind{4, 6, \infty} = 0 = \ind{7, 4, 6}\)};

		\begin{scope}[xshift=22cm]
			\node at (-1.9,4.5) {$\psi(T)=$};
			\draw (1, 7) rectangle ++(1, 1);
			\draw (1, 6) rectangle ++(1, 1)[fill=gray!50];
			\draw (0, 6) rectangle ++(1, 1)[fill=gray!50];
			\draw (1, 5) rectangle ++(1, 1)[fill=gray!50];
			\draw (0, 5) rectangle ++(1, 1)[fill=gray!50];
			\draw (1, 4) rectangle ++(1, 1);
			\draw (0, 4) rectangle ++(1, 1);
			\draw (1, 3) rectangle ++(1, 1);
			\draw (0, 3) rectangle ++(1, 1);
			\draw (1, 2) rectangle ++(1, 1);
			\draw (0, 2) rectangle ++(1, 1);
			\draw (1, 1) rectangle ++(1, 1)[fill=gray!50];
			\draw (0, 1) rectangle ++(1, 1)[fill=gray!50];
			\node at (1.5, 7.5) {\(3\)};
			\node at (1.5, 6.5) {\(5\)};
			\node at (0.5, 6.5) {\(2\)};
			\node at (1.5, 5.5) {\(1\)};
			\node at (0.5, 5.5) {\(4\)};
			\node at (1.5, 4.5) {\(5\)};
			\node at (0.5, 4.5) {\(8\)};
			\node at (1.5, 3.5) {\(1\)};
			\node at (0.5, 3.5) {\(6\)};
			\node at (1.5, 2.5) {\(3\)};
			\node at (0.5, 2.5) {\(7\)};
			\node at (1.5, 1.5) {\(4\)};
			\node at (0.5, 1.5) {\(6\)};
		\end{scope}
	\end{tikzpicture}
	\caption{Columns \(i, i+1\) of the fillings \(\phi(T)\) and $\psi(T)$ derived from \(T\) in \cref{fig:blocks}.}
	\label{fig:output-filling}
\end{figure}

\begin{theorem}[{\cite[Proposition~4.3]{KLO22}}] \label{thm:phi}
	The map \(\phi \colon \Fill(\alpha) \to \Fill(s_i \cdot \alpha)\) is a bijection satisfying \(\mathbf{x}^{|T|} = \mathbf{x}^{|\phi(T)|}\), \(\maj(T) = \maj(\phi(T))\), and \(\inv(T) = \inv(\phi(T))\) for all \(T \in \Fill(\alpha)\).
\end{theorem}

As an auxiliary construction, define \(\psi(T)\) to be the filling obtained from \(\phi(T)\) by swapping the bottommost block.
The map \(\psi\) changes the statistics \(\inv\) and \(\maj\) in a controlled way, as follows.

\begin{theorem}[{\cite[17]{KLO22}}] \label{thm:psi}
	Assume \(\alpha_i > \alpha_{i+1}\).
	The map \(\psi \colon \Fill(\alpha) \to \Fill(s_i \cdot \alpha)\) is a bijection such that, for all \(T \in \Fill(\alpha)\) satisfying \(\ind{{T(i, 1)},{T(i+1, 1)}} = 1\), we have \(\mathbf{x}^{|T|} = \mathbf{x}^{|\psi(T)|}\), and
	\begin{equation*}
		(\inv(T),\maj(T)) =
		\begin{cases}
			(\inv(\psi(T)) + A,\,\maj(\psi(T)) - (L+1)) & \text{if } \ind{{\psi(T)(i, 1)},{\psi(T)(i+1, 1)}} = 1\,, \\
			(\inv(\psi(T)) + 1,\,\maj(\psi(T)))         & \text{if } \ind{{\psi(T)(i, 1)},{\psi(T)(i+1, 1)}} = 0\,.
		\end{cases}
	\end{equation*}
\end{theorem}

\begin{remark}
	\textcite{KLO22} define $\phi$ and $\psi$ by a recursive procedure that examines the filling from the bottom row upward and may invoke either map at intermediate steps.
	Their definition is less transparent, but is well suited to induction.
	Their recursive constructions produce the same bijections as our nonrecursive descriptions.
\end{remark}

\begin{remark}
	As a corollary, we get a simple bijective proof of \cite[Theorem~5.1.1]{HHL08}.
\end{remark}

\subsection{Shape changing identity}

In this subsection, we establish \cref{theorem:three-term}.
Each permuted-basement Macdonald polynomial admits a signed filling expansion as in \cref{lem:signed}, thus \eqref{eq:three term} is equivalent to
\begin{equation} \label{eq:Fill pm}
	\sum_{T\in\Fill_\pm(\alpha,\sigma)}\wt_\pm(T)
	=
	\sum_{T\in\Fill_\pm(s_i\alpha,\sigma s_i)}\wt_\pm(T)
	+
	c_{i,\alpha,\sigma}(q, t) \sum_{T\in\Fill_\pm(s_i\alpha,\sigma)}\wt_\pm(T).
\end{equation}
We relate each element of the left-hand sum to a pair of elements in the right-hand sums, using \cref{thm:phi,thm:psi} as building blocks.

Using the signed $\chi$ function, we extend \(\phi\) and \(\psi\) to maps on signed augmented fillings \(\Fill_\pm(\alpha,\sigma)\).
Because each map may or may not interchange the \(i\)\textsuperscript{th} and \((i+1)\)\textsuperscript{th} basement entries, their natural domain is
\(\Fill_\pm(s_i\alpha,\sigma)\,\cup\,\Fill_\pm(s_i\alpha,\sigma s_i).\)
Nevertheless, \(\psi(T)\) is obtained from \(\phi(T)\) by swapping the bottommost block, so exactly one of \(\phi(T)\) and \(\psi(T)\) lies in \(\Fill_\pm(s_i\alpha,\sigma s_i)\), and the other lies in \(\Fill_\pm(s_i\alpha,\sigma)\).
Hence for each \(T\) there exist unique
\begin{equation*}
	f(T)\in \Fill_\pm(s_i\alpha,\sigma s_i),\qquad
	g(T)\in \Fill_\pm(s_i\alpha,\sigma)
\end{equation*}
such that \(\{f(T),g(T)\}=\{\phi(T),\psi(T)\}\).
The bijections \(f\) and \(g\) associate, to each \(T\in\Fill_\pm(\alpha,\sigma)\), a pair of fillings appearing on the right side of \eqref{eq:Fill pm}.
\cref{lemma:T-fT-gT}, proved using \cref{thm:psi}, asserts that their weights match the required contributions in \eqref{eq:Fill pm}.

\begin{lemma} \label{lemma:T-fT-gT}
	For all \(T \in \Fill_\pm(\alpha, \sigma)\), we have
	\begin{equation*}
		\wt_\pm(T)
		=
		\wt_\pm(f(T))
		+
		c_{i, \alpha, \sigma}(q, t)
		\wt_\pm(g(T)),
	\end{equation*}
	where \(c_{i, \alpha, \sigma}(q, t)\) is as in \cref{theorem:three-term}.
\end{lemma}

Using \cref{lemma:T-fT-gT} and summing over all \(T \in \Fill_\pm(\alpha, \sigma)\),  we obtain \eqref{eq:Fill pm} and hence \cref{theorem:three-term}.

\begin{example} \label{example:T-fT-gT}
	For \(\alpha = \composition{3, 2}\), \(\sigma = \window{2, 1}\), and \(i = 1\), we have \(c_{i, \alpha, \sigma}(q, t) = \frac{1 - t}{1 - q^{1} t^{1}}\).
	Below,
	we show a filling \(T\in\Fill_\pm(\alpha,\sigma)\) and its derived fillings \(f(T)\) and \(g(T)\).
	We compute \(\wt_\pm(T) = \frac{x_1^3x_2^2 q^4 t^1 (-t)^3}{(1-qt)P(q, t)}\),
	\(\wt_\pm(f(T)) = \frac{x_1^3x_2^2 q^4 t^2 (-t)^3}{(1-qt^2)P(q, t)}\), and
	\(\wt_\pm(g(T)) = \frac{x_1^3x_2^2 q^4 t^1 (-t)^3}{(1-qt^2)P(q, t)}\),
	where $P(q, t) = {\scriptstyle (1-qt)(1-q^2t)(1-q^2t^2)(1-q^2t^3)}$.
	We confirm \cref{lemma:T-fT-gT} for this \(T\) by checking that
	\begin{equation*}
		\frac{x_1^3x_2^2 q^4 t^1 (-t)^3}{(1-qt)P(q, t)}
		=
		\frac{x_1^3x_2^2 q^4 t^2 (-t)^3}{(1-qt^2)P(q, t)}
		+
		\frac{1 - t}{1 - q^{1} t^{1}}\;
		\frac{x_1^3x_2^2 q^4 t^1 (-t)^3}{(1-qt^2)P(q, t)}.
	\end{equation*}
\end{example}

\begin{center}
	\begin{tikzpicture}[scale=.5]
		\begin{scope}
			\node[anchor=east] at (-.1,3) {$T=$};
			\draw (0, 4) rectangle ++(1, 1);
			\draw (1, 3) rectangle ++(1, 1);
			\draw (0, 3) rectangle ++(1, 1);
			\draw (1, 2) rectangle ++(1, 1);
			\draw (0, 2) rectangle ++(1, 1);
			\draw (1, 1) rectangle ++(1, 1);
			\draw (0, 1) rectangle ++(1, 1);
			\node at (0.5, 4.5) {\(\overline{2}\)};
			\node at (0.5, 3.5) {\(\overline{1}\)};
			\node at (1.5, 3.5) {\(2\)};
			\node at (0.5, 2.5) {\(\overline{1}\)};
			\node at (1.5, 2.5) {\(1\)};
			\node at (0.5, 1.5) {\(2\)};
			\node at (1.5, 1.5) {\(1\)};
		\end{scope}
		\begin{scope}[xshift=10cm]
			\node[anchor=east] at (-.1,3) {$f(T)=\psi(T)=$};
			\draw (1, 4) rectangle ++(1, 1);
			\draw (1, 3) rectangle ++(1, 1);
			\draw (0, 3) rectangle ++(1, 1);
			\draw (1, 2) rectangle ++(1, 1);
			\draw (0, 2) rectangle ++(1, 1);
			\draw (1, 1) rectangle ++(1, 1);
			\draw (0, 1) rectangle ++(1, 1);
			\node at (1.5, 4.5) {\(\overline{2}\)};
			\node at (0.5, 3.5) {\(\overline{1}\)};
			\node at (1.5, 3.5) {\(2\)};
			\node at (0.5, 2.5) {\(1\)};
			\node at (1.5, 2.5) {\(\overline{1}\)};
			\node at (0.5, 1.5) {\(1\)};
			\node at (1.5, 1.5) {\(2\)};
		\end{scope}
		\begin{scope}[xshift=20cm]
			\node[anchor=east] at (-.1,3) {$g(T)=\phi(T)=$};
			\draw (1, 4) rectangle ++(1, 1);
			\draw (1, 3) rectangle ++(1, 1);
			\draw (0, 3) rectangle ++(1, 1);
			\draw (1, 2) rectangle ++(1, 1);
			\draw (0, 2) rectangle ++(1, 1);
			\draw (1, 1) rectangle ++(1, 1);
			\draw (0, 1) rectangle ++(1, 1);
			\node at (1.5, 4.5) {\(\overline{2}\)};
			\node at (0.5, 3.5) {\(\overline{1}\)};
			\node at (1.5, 3.5) {\(2\)};
			\node at (0.5, 2.5) {\(\overline{1}\)};
			\node at (1.5, 2.5) {\(1\)};
			\node at (0.5, 1.5) {\(2\)};
			\node at (1.5, 1.5) {\(1\)};
		\end{scope}
	\end{tikzpicture}
\end{center}

\begin{remark}
	The inversion number for an augmented filling of shape \(\alpha\) with basement \(\sigma\) and the inversion number for its associated non-augmented filling are not the same; they differ by \(\twinv(\alpha,\sigma)\).
	This discrepancy explains why the exponent \(\twinv(s_i \alpha, \sigma) - \twinv(\alpha, \sigma)\) appears in \(c_{i, \alpha, \sigma}(q, t)\).
\end{remark}

\begin{remark}
	Instead of using the signed filling formula of \cref{lem:signed} to prove \cref{theorem:three-term}, one could also attempt to work directly with the non-attacking fillings appearing in the compact expression of \cref{thm:tableau-formula}.
	One approach is to adapt the recursive construction for $\phi$ and $\phi$ from \cite{KLO22}, introducing suitable weights to produce a probabilistic bijection on non-attacking fillings,
	although we do not develop this idea here.
\end{remark}

\section{Straightening} \label{sec:straightening}

\begin{definition}
	For compositions $\alpha$ and $\beta$ that are rearrangements of each other,
	let $\pi_\alpha, \pi_\beta \in\sym{n}$ be the shortest permutations such that
	$\pi_\alpha \cdot \alpha=\inc(\alpha)$ and $\pi_\beta \cdot \beta=\inc(\beta)$.
	Define the \vocab{Bruhat order on compositions} by $\alpha \succ \beta$
	if and only if $\pi_\alpha > \pi_\beta$ in Bruhat order.
\end{definition}

The identity of \cref{theorem:three-term} generalizes to a \emph{straightening rule}, allowing us to expand $E_\alpha^\sigma$ in the polynomials $\{E_{\tau\alpha}^\theta\}_\theta$ whenever $\alpha\succ \tau\alpha$.
Moreover,  we obtain an explicit expression for the coefficients by keeping track of when the first or second term of the right hand side of \eqref{eq:three term} arises.

For a reduced word $\tau=s_{i_k}\cdots s_{i_1}$ and a subset $I=\{j_1,\ldots,j_\ell\}\subseteq [k]$, define  $\tau_I\coloneqq s_{i_{j_\ell}}\cdots s_{i_{j_1}}$ to be the corresponding subword of $\tau$, indexed by $I$.
Define also the $j$-truncated index set $I^{<j}=I\cap [j-1]$ and the $j$-truncated subword $\tau^{<j}_I\coloneqq \tau_{I^{<j}}$ indexed by this set. Notice that this corresponds to the sequence of transpositions indexed by $I$, contained within the first $j-1$ elements of $\tau$, from the right.

\begin{proposition} \label{prop:expansion}
	Let $\alpha$ be a composition and let $\beta=\tau\alpha$ for a reduced word $\tau=s_{i_k}\cdots s_{i_1}$ satisfying $\alpha\succ\beta$. Then,
	\begin{equation*}
		E_\alpha^\sigma(\mathbf{x};q,t)=\sum_{I\subseteq[k]}\Phi(\alpha,\sigma;\tau,I)E_{\beta}^{\sigma\tau_I}(\mathbf{x};q,t),\qquad \Phi(\alpha,\sigma;\tau,I)=\prod_{j\not\in I}\Psi(\alpha^{(j)},\sigma \tau^{<j}_I,i_j).
	\end{equation*}
	where $\Psi(\beta,\rho,i)=c_{i,\beta,\rho}(q,t)$ as given in \cref{theorem:three-term}.
\end{proposition}

The coefficients in the expansion above are sums of terms of the form
\begin{equation*}
	q^at^b\prod_{(i,j)\in D}\frac{1-t}{1-q^it^j},\qquad a,b\in\nonnegatives,\qquad D\subset\positives^2\,,
\end{equation*}
where $D$ is a finite multiset bounded by the box containing $\dg(\alpha)$. At $t=0$, these coefficients are in $\nonnegatives[q]$, and at $q=0$, the expansion is a non-negative integer linear combination of monomials of the form $t^a(1-t)^b$.
\begin{corollary} \label{cor:E t=0}
	Let $\alpha$ and $\beta$ be a compositions with $\alpha\succ\beta$. Then,
	\begin{equation*}
		E_\alpha^\sigma(\mathbf{x};q,0)=\sum_{\theta}d_{\alpha,\sigma,\theta}(q)E_{\beta}^{\theta}(\mathbf{x};q,0),\qquad d_{\alpha,\sigma,\theta}(q)\in\nonnegatives[q].
	\end{equation*}
\end{corollary}

In particular, since we always have $\alpha\succeq\inc(\alpha)$, the straightening rule yields a ``nice'' formula for the expansion of $E_\alpha^\sigma$ in the ASEP polynomials $\{F_\beta\}_{\beta\in \sym{n}(\alpha)}=\{E_{\inc(\alpha)}^\theta\}_{\theta}$, with coefficients in the form described above.
\begin{question}
	Find a combinatorial interpretation for the coefficients $d_{\alpha, \sigma, \beta}(q)$ in \cref{cor:E t=0}.
\end{question}

\printbibliography

\end{document}